\documentclass[12pt,reqno]{amsart}
\usepackage{amsfonts, amsthm, amsmath}
\usepackage{rotating}
\usepackage{tikz}
\usepackage{graphics}
\usepackage{amssymb}
\usepackage{amscd}
\usepackage[latin2]{inputenc}
\usepackage{t1enc}
\usepackage[mathscr]{eucal}
\usepackage{ferrers}
\usepackage{indentfirst}
\usepackage{graphicx}
\usepackage{graphics}
\usepackage{pict2e}
\usepackage{mathrsfs}
\usepackage{enumerate}
\usepackage[pagebackref]{hyperref}
\hypersetup{colorlinks=true}
\usepackage{cite}
\usepackage{color}
\usepackage{epic}
\usepackage{hyperref} 
\numberwithin{equation}{section}
\theoremstyle{plain}

\newtheorem{theorem}{Theorem}[section]
\newtheorem{lemma}[theorem]{Lemma}
\newtheorem{corollary}[theorem]{Corollary}

\newtheorem{proposition}[theorem]{Proposition}

\newtheorem{observation}[theorem]{Observation}
\newtheorem{conjecture}[theorem]{Conjecture}
\theoremstyle{definition}
\newtheorem{Def}[theorem]{Definition}

\newtheorem{remark}[theorem]{Remark}
\newtheorem{?}[theorem]{Problem}

\def\Z{\mathbb{Z}}
\def\cO{\mathcal{O}}
\def\cE{\mathcal{E}}
\def\OE{\mathcal{OE}}
\def\EO{\mathcal{EO}}
\def\fEO{\mathfrak{EO}}

\def\fF{\mathfrak{F}}
\def\bOE{\overline{\mathcal{OE}}}
\def\bEO{\overline{\mathcal{EO}}}
\def\bfEO{\overline{\mathfrak{EO}}}
\def\bfOE{\overline{\mathfrak{OE}}}

\def\eoc{\mathrm{eoc}}

\def\fP{\mathfrak{P}}

\def\boxit#1{\leavevmode\hbox{\vrule\vtop{\vbox{\kern.33333pt\hrule
    \kern1pt\hbox{\kern1pt\vbox{#1}\kern1pt}}\kern1pt\hrule}\vrule}}

\usepackage{collectbox}



\begin{document}

\title[Partitions separated by parity]
{Partitions with parts separated by parity: conjugation, congruences and the mock theta functions}

\author[S. Fu]{Shishuo Fu}
\address[Shishuo Fu]{College of Mathematics and Statistics, Chongqing University, Huxi Campus,
Chongqing 401331, P.R. China}
\email{fsshuo@cqu.edu.cn}

\author[D. Tang]{Dazhao Tang}
\address[Dazhao Tang]{School of Mathematical Sciences, Chongqing Normal University,
Chongqing 401331, P.R. China}
\email{dazhaotang@sina.com}

\date{\today}

\begin{abstract}
Noting a curious link between Andrews' even-odd crank and the Stanley rank, we adopt a combinatorial approach
building on the map of conjugation and continue the study of integer partitions with parts separated by
parity. Our motivation is twofold. First off, we derive results for certain restricted partitions with even
parts below odd parts. These include a Franklin-type involution proving a parametrized identity that
generalizes Andrews' bivariate generating function, and two families of Andrews--Beck type congruences.
Secondly, we introduce several new subsets of partitions that are stable (i.e., invariant under conjugation)
and explore their connections with three third order mock theta functions $\omega(q)$, $\nu(q)$, and
$\psi^{(3)}(q)$, introduced by Ramanujan and Watson.
\end{abstract}



\maketitle



\section{Introduction}\label{sec:intro}

A \emph{partition} $\lambda=(\lambda_1,\ldots,\lambda_r)$ of a positive integer $n$ is a finite weakly
decreasing sequence of positive integers $\lambda_{1}\geq\lambda_{2}\geq\cdots\geq\lambda_{r}\ge 1$ such that
$\sum_{i=1}^{r}\lambda_{i}=n$, denoted as $\lambda\vdash n$. The $\lambda_{i}$ are called the \emph{parts} of
the partition $\lambda$. As usual, we denote the number being partitioned $n$, and the {\em number of parts}
$r$, as $|\lambda|$ and $\#(\lambda)$, respectively. The partition function $p(n)$ is the number of
partitions of $n$. For the sake of convenience, we denote the empty partition of $0$ as $\epsilon$ and agree
that it is contained in the set of ordinary partitions and various subclasses of restricted partitions. Hence
for example, $p(0)=1$. Unless otherwise noted, we will follow the notations used in \cite{and}.

In 1944, Dyson \cite{Dys} defined the \emph{rank} of a partition as the largest part minus the number of
parts, and then observed that the rank appears to give combinatorial interpretations for the first two (i.e.,
\eqref{cong:mod5} and \eqref{cong:mod7}) of Ramanujan's celebrated partition congruences, namely,
\begin{align}
p(5n+4) &\equiv0\pmod{5},\label{cong:mod5}\\
p(7n+5) &\equiv0\pmod{7},\label{cong:mod7}\\
p(11n+6) &\equiv0\pmod{11}.\label{cong:mod11}
\end{align}
Since the map of conjugation (see Sect.~\ref{sec:pre} for definition) swaps the largest part and the number
of parts between a partition $\lambda$ and its conjugate $\lambda'$, Dyson's rank can be rephrased as
\begin{align*}
\textrm{rank}(\lambda)=\#(\lambda')-\#(\lambda).
\end{align*}

A similar looking partition statistic, the so-called {\em Stanley rank}, was introduced by Stanley \cite{sta}
in his study of sign-balanced, labelled posets (see also \cite{and1, BG1, BG2}), and can be defined as
\begin{align}
\textrm{srank}(\lambda)=\mathcal{O}(\lambda)-\mathcal{O}(\lambda'),
\end{align}
where $\mathcal{O}(\lambda)$ is the number of odd parts in $\lambda$. This statistic also refines Ramanujan's
first congruence \eqref{cong:mod5}; see \cite[Corollary~1.2]{and1}.

Our starting point is an observation that connects the Stanley rank with Andrews' recent work \cite{and2} on
partitions with even parts below odd parts. We need some further definitions before we can state this
observation.

Let $\mathfrak{EO}_n$ denote the set of partitions of $n$ in which every even part is less than each odd
part, and let $\mathcal{EO}(n):=|\mathfrak{EO}_n|$, $\mathfrak{EO}:=\bigcup_{n\ge 0}\mathfrak{EO}_n$. Let
$\overline{\mathfrak{EO}}_n$ denote the set of partitions in $\mathfrak{EO}_n$ in which ONLY the largest even
part, if any, appears an odd number of times, then $\overline{\mathcal{EO}}(n)$ and
$\overline{\mathfrak{EO}}$ can be defined analogously. More generally, if $\mathfrak{P}$ denote a set of
partitions with certain restrictions, then $\mathfrak{P}_n:=\{\lambda\in\fP:|\lambda|=n\}$, and
$\mathcal{P}(n):=|\mathfrak{P}_n|$.

Recall the third order mock theta function due to Ramanujan and Watson (see, for example, \cite{ADY}),
defined by
\begin{align}
\nu(q) &=\sum_{n=0}^{\infty}\frac{q^{n^2+n}}{(-q;q^2)_{n+1}},\label{def:nu}\\
\omega(q) &=\sum_{n=0}^{\infty}\frac{q^{2n^2+2n}}{(q;q^2)_{n+1}^2},\label{def:omega}
\end{align}
where, here and throughout the rest of this paper, we always assume that $q$ is a complex number such that
$|q|<1$ and adopt the following customary abbreviations in partitions and $q$-series:
\begin{align*}
(a;q)_n:=\prod_{j=0}^{n-1}(1-aq^j)\qquad\textrm{and}\qquad(a;q)_\infty:=\prod_{j=0}^\infty(1-aq^j).
\end{align*}
Noting the connection with the third order mock theta function $\nu(q)$, Andrews deduced in \cite{and2}
that
\begin{align}\label{cong:bEO}
\overline{\mathcal{EO}}(10n+8)\equiv0\pmod{5}.
\end{align}
In order to explain \eqref{cong:bEO} combinatorially as the rank does for \eqref{cong:mod5} and
\eqref{cong:mod7}, he went on to introduce the \textit{even-odd crank} for any partition
$\lambda\in\mathfrak{EO}$, to be
\begin{align}\label{eo-crank}
\textrm{eoc}(\lambda)=\textrm{the~largest~even~part~of}~\lambda-\mathcal{O}(\lambda).
\end{align}
Of course, $\textrm{eoc}(\lambda)$ is always an even number when $\lambda\in\overline{\mathfrak{EO}}$. Let
$N_{\textrm{eo}}(k,m,n)$ denote the number of partitions in $\overline{\mathfrak{EO}}_n$ whose even-odd
cranks are congruent to $k$ modulo $m$. Andrews \cite[Theorem~3.3]{and2} proved that for $0\leq i\leq 4$,
\begin{align}
N_{\textrm{eo}}(i,5,10n+8)=\frac{1}{5}\overline{\mathcal{EO}}(10n+8).\label{eq:1/5}
\end{align}

The key observation mentioned earlier is as follows.

\begin{observation}\label{ob:key}
For any partition $\lambda$, $\lambda\in\overline{\mathfrak{EO}}$ if and only if
$\lambda'\in\overline{\mathfrak{EO}}$. In this case, we have
\begin{align}\label{eq:Andrews=Stanley}
\emph{eoc}(\lambda)=\mathcal{O}(\lambda')-\mathcal{O}(\lambda)=-\emph{srank}(\lambda).
\end{align}
\end{observation}

This link sheds new light on the study of $\overline{\mathcal{EO}}(n)$, since now we can utilize the symmetry
of $\overline{\mathfrak{EO}}$, imposed by the map of conjugation.

Motivated by \eqref{cong:mod5}, \eqref{cong:mod7} and the statistic rank, Beck proposed some surprisingly
conjectural congruences. Let $NT(k,m,n)$ denote the total number of parts in the partitions of $n$ with rank
congruent to $k$ modulo $m$, Andrews \cite[Theorems 1 and 2]{and4} proved for any $n\geq0$,
\begin{align}
 &\big(NT(1,5,5n+i)-NT(4,5,5n+i)\big)\nonumber\\
 &\quad+2\big(NT(2,5,5n+i)-NT(3,5,5n+i)\big)\equiv0\pmod{5},\quad\textrm{if}
\;\;i=1\;\textrm{or}\;4,\label{AB-cong-mod5}\\
 &\big(NT(1,7,7n+j)-NT(6,7,7n+j)\big)\nonumber\\
 &\quad+\big(NT(2,7,7n+j)-NT(5,7,7n+j)\big)\nonumber\\
 &\quad-\big(NT(3,7,7n+j)-NT(4,7,7n+j)\big)\equiv0\pmod{7},\quad\textrm{if}
\;\;j=1\;\textrm{or}\;5.\label{AB-cong-mod7}
\end{align}
The congruences \eqref{AB-cong-mod5} and \eqref{AB-cong-mod7} are in general called Andrews--Beck type
congruences. Motivated by these two congruences, many authors have recently discovered Andrews--Beck type
congruences for various types of partitions with their associated statistics; see, for example,
\cite{Che1, Che22a, Che22b, CMO, DT22a, DT22b, JLX22, Kim22, LPT21, Mao22a, Mao22b, Mao23, Yao22}. Let
$NT_{\textrm{eo}}(k,m,n)$ denote the total number of odd parts among partitions in
$\overline{\mathfrak{EO}}_n$, whose even-odd cranks are congruent to $k$ modulo $m$. Our first result is the
following Andrews--Beck type congruences. Unlike the methods used in the aforementioned literature, the
proofs of \eqref{AB-cong:10n8} and \eqref{AB-cong:10n} need to take advantage of the map of conjugation and
\eqref{eq:Andrews=Stanley}.

\begin{theorem}\label{THM:AB-cong}
For any $n\geq0$,
\begin{align}
&\big(NT_{\emph{eo}}(1,5,10n)-NT_{\emph{eo}}(4,5,10n)\big)\notag\\
&\quad+2\big(NT_{\emph{eo}}(2,5,10n)-NT_{\emph{eo}}(3,5,10n)\big)\equiv0\pmod{10},\label{AB-cong:10n8}\\
&\big(NT_{\emph{eo}}(1,5,10n+8)-NT_{\emph{eo}}(4,5,10n+8)\big)\notag\\
&\quad+2\big(NT_{\emph{eo}}(2,5,10n+8)-NT_{\emph{eo}}(3,5,10n+8)\big)\equiv0\pmod{20}.\label{AB-cong:10n}
\end{align}
\end{theorem}

In a follow-up paper \cite{and3}, Andrews further studied various types of partitions with parts separated by
parity. Motivated by this work, we also consider some subclasses of partitions with parts separated by
parity. Let $\mathcal{OE}(n)$ denote the number of partitions of $n$ in which each odd part is less than each
even part. Denote by $\overline{\mathcal{OE}}(n)$ the number of partitions counted by $\OE(n)$ in which both
even and odd parts appear, and ONLY the largest odd part and the largest even part appear an odd number of
times, so in particular, $\overline{\mathcal{OE}}(n)>0$ only for odd $n\ge 3$. For example $\bOE(9)=8$, with
the eligible partitions being $(8,1)$, $(6,3)$, $(6,1,1,1)$, $(4,3,1,1)$, $(4,2,2,1)$, $(4,1,1,1,1,1)$,
$(2,2,2,1,1,1)$, and $(2,1,1,1,1,1,1,1)$.

Relating to the Stanley rank, we define $\bOE(m,n)$  to be the number of partitions enumerated by $\bOE(n)$
whose Stanley rank equals $m$ and let
\begin{align*}
\overline{\textrm{OE}}(z,q):=\sum_{\pi\in\overline{\mathfrak{OE}}}z^{\textrm{srank}(\pi)}q^{|\pi|}
=\sum_{n=0}^\infty\sum_{m=-\infty}^{\infty}\overline{\mathcal{OE}}(m,n)z^m q^n
\end{align*}
be the bivariate generating function for $\overline{\mathcal{OE}}(m,n)$.

\begin{theorem}\label{thm:bOE-srank-gf}
We have
\begin{align}\label{gf:boe}
\overline{\emph{OE}}(z,q)=\sum_{m=1}^{\infty}\sum_{n=1}^{\infty}
\dfrac{q^{4mn-1}}{(z^2q^2;q^4)_{m}(q^2/z^2;q^4)_{n}}.
\end{align}
\end{theorem}

The rest of this paper is organized as follows. We introduce the pivotal concept of ``stable'' sets of
partitions in Section \ref{sec:pre} (see Definition \ref{def:stable}), after which Observation \ref{ob:key} is
explained, and a complete characterization of the parity of $\bEO(n)$ is rederived with ease. Theorem
\ref{THM:AB-cong} is proved in Section \ref{sec:EO}, where three further stable subsets contained in $\fEO$ and
their connections with mock theta functions are considered as well. In Section \ref{sec:OE}, we prove
Theorem \ref{thm:bOE-srank-gf} and explore further subsets of $\bfOE$. We raise three conjectures to conclude
the paper.

\section{Preliminaries}\label{sec:pre}
We set further notations and lay the groundwork in this section.

To each partition $\lambda$, we associate a graphical representation called {\em Ferrers graph}, which is a
left-justified array of unit squares such that the $i$-th row contains $\lambda_i$ squares. Ferrers graph
facilitates the illustration of the following three notions that will appear frequently in this paper (see
Fig.~\ref{ferrers}). Suppose we are given a partition $\lambda=(\lambda_1,\ldots,\lambda_r)$.
\begin{itemize}
\item Conjugation: The {\em conjugate} of $\lambda$ is obtained by transposing its Ferrers graph.
Equivalently, the conjugate partition of $\lambda$ is $\lambda'=(\lambda_1',\ldots,\lambda_{\lambda_1}')$,
where $\lambda_i'=|\{1\le j\le r:\lambda_j\ge i\}|$.

\item Durfee square: The {\em Durfee square} of $\lambda$ is the largest square that can fit into the
top-left corner of its Ferrers graph. If we denote the {\em side of the Durfee square} of $\lambda$ as
$d(\lambda)$, then $d(\lambda)=|\{1\le j\le r:\lambda_j\ge j\}|$. Note that $\lambda$ has a unique
decomposition $\lambda=(d(\lambda);\mu,\nu)$, with $\mu$ and $\nu$ being the subpartition below the Durfee
square of $\lambda$ and the conjugate of the subpartition to the right of the Durfee square, respectively.
Clearly $|\lambda|=d^2+|\mu|+|\nu|$.

\item Profile: The {\em profile} of a nonempty partition $\lambda$ is the sequence of southmost and
eastmost border edges in its Ferrers graph, that starts with an east edge and ends with a north edge. We
use the {\em profile word} to record the profile of $\lambda$: $w_{\lambda}=(e_1,n_1,\ldots,e_k,n_k)$,
where for each $1\le i\le k\le \min\{r,\lambda_1\}$, $e_i$ and $n_i$ are the numbers of consecutive east
and north edges appearing alternately in the profile, respectively. It is worth noting that Dyson's rank
can be rephrased as $\textrm{rank}(\lambda)=\sum_i e_i-\sum_i n_i$.
\end{itemize}

\begin{figure}[htp]
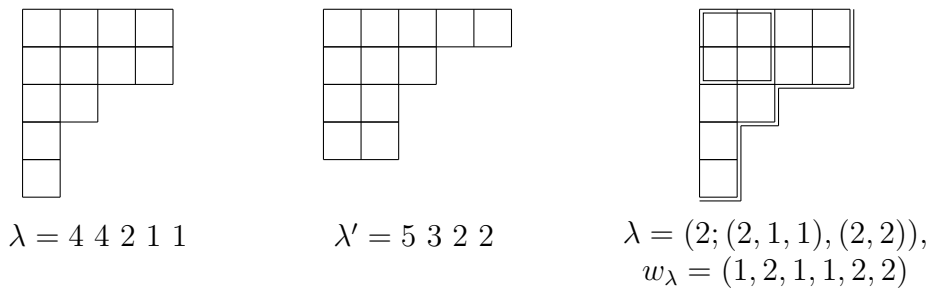

\begin{ferrers}
\addcellrows{4+4+2+1+1}
\addtext{1}{-3}{$\lambda=4~4~2~1~1$}
\setbase{4}{0}
\addcellrows{5+3+2+2}
\addtext{1.2}{-3}{$\lambda'=5~3~2~2$}
\setbase{9}{0}
\addcellrows{4+4+2+1+1}
\addline{0}{-2.55}{0.55}{-2.55}
\addline{0.55}{-1.55}{0.55}{-2.55}
\addline{0.55}{-1.55}{1.05}{-1.55}
\addline{1.05}{-1.55}{1.05}{-1.05}
\addline{1.05}{-1.05}{2.05}{-1.05}
\addline{2.05}{-1.05}{2.05}{0}
\highlightdurfeesquarebycell{2}
\addtext{1}{-3}{$\lambda=(2;(2,1,1),(2,2))$,}
\addtext{1}{-3.5}{$w_{\lambda}=(1,2,1,1,2,2)$}
\end{ferrers}
\caption{Ferrers graph, conjugation, Durfee square and profile}
\label{ferrers}
\end{figure}

Note that both the conjugation and the Durfee square are standard notions that have been widely applied in
the study of partition theory \cite{and}. While the partition profile, whose definition follows \cite{KN},
was involved in our previous work \cite{FT}; see also \cite{Fu22, LXY22} for recent works demonstrating the
usefulness of this alternative way of recording a partition. A closely related notion, the {\em rim hook},
makes an appearance in the famed Murnaghan--Nakayama rule (see, for example, \cite[Theorem~4.10.2]{sag})
and is crucial in the study of symmetric functions in general.

As a standard tool in the realm of enumerative combinatorics, the theory of enumeration under group action,
or \textit{P\'olya theory}, seems to be lack of application in the study of integer partitions.
Nonetheless, one notable and enlightening exception was the paper by Garvan, Kim and Stanton \cite{GKS},
in which they utilized three dihedral groups' actions on the partitions of $5n+4, 7n+5$, and $11n+6$, to
give a uniform proof of all three Ramanujan congruences \eqref{cong:mod5}, \eqref{cong:mod7} and
\eqref{cong:mod11}; see also Hirschhorn's paper \cite{hir} and book \cite[Chap.~4]{hirb} for a more direct
version of their proof. We would also like to mention the recent work of Kim \cite{kim}, in which a
$\Z_2$-action similar to conjugation was discussed and interesting congruences were derived.

Motivated by these previous works and our Observation \ref{ob:key}, we view the map of conjugation as a
$\Z_2$-action on the set of partitions. From this perspective, it is natural to consider those subsets that
are invariant under this group action. We give a name to such subsets.

\begin{Def}\label{def:stable}
A collection of partitions, say $\mathfrak{C}$, is said to be {\em (conjugational) stable}, if the
following holds true:
\begin{align*}
\lambda\in\mathfrak{C}\quad\textrm{if~and~only~if}\quad\lambda'\in\mathfrak{C}.
\end{align*}
\end{Def}

For instance, we state without proof in Observation \ref{ob:key} that $\overline{\mathfrak{EO}}$ is
stable. To see this, simply note that for any $\lambda\in\overline{\mathfrak{EO}}$, its profile word $w_{\lambda}=(e_1,n_1,\ldots,e_k,n_k)$ must contain a consecutive pair of odd numbers $(n_i,e_{i+1})$ for
some $0\le i\le k$ (set $n_0=e_{k+1}=1$ by convention), and all remaining letters in $w_{\lambda}$ are even.
This property is clearly seen to be preserved by the reversal, thus $\lambda'\in\overline{\mathfrak{EO}}$ as
well (see Proposition~\ref{prop:conj and rev} below). Furthermore, we see that the largest even part of
$\lambda$ is given by $e_1+e_2+\cdots+e_i$, which, upon reversal, outputs exactly $\cO(\lambda')$, the
number of odd parts in $\lambda'$, hence \eqref{eq:Andrews=Stanley} follows and Observation \ref{ob:key}
holds true.

From now on, we will use stability as a guideline for discovering new subclasses of partitions with parts
separated by parity. For any class of partitions $\mathfrak{C}$, we use $\mathfrak{C}_{c}$ to denote the
subset of partitions $\lambda\in\mathfrak{C}$ that are {\em self-conjugate}, i.e., $\lambda=\lambda'$. The
following two facts are simple but useful.

\begin{proposition}\label{prop:conj and rev}
For a partition $\lambda$ with its profile word $w_{\lambda}=(e_1,n_1,\ldots,e_k,n_k)$, the profile word for
its conjugate is obtained from reversing $w_{\lambda}$, i.e.,
\begin{align*}
w_{\lambda'}=(n_k,e_k,\ldots,n_1,e_1).
\end{align*}
In particular, $\lambda$ is self-conjugate if and only if $w_{\lambda}$ is palindromic, i.e., $e_1=n_k$,
$n_1=e_k,\ldots$.
\end{proposition}

\begin{proposition}\label{Prop:conj}
Let $\mathcal{C}(n):=|\{\lambda\in\mathfrak{C}:|\lambda|=n\}|$ and
$\mathcal{C}_c(n):=|\{\lambda\in\mathfrak{C}_c:|\lambda|=n\}|$. Then
\begin{align}\label{cong:sc-parity}
\mathcal{C}(n)\equiv\mathcal{C}_c(n)\pmod 2.
\end{align}
\end{proposition}

Based on Proposition \ref{Prop:conj}, we give another proof of the following complete characterization for
the parity of $\overline{\mathcal{EO}}(n)$, first derived by Passary \cite[Eq. (2.1.24)]{Pas}.

\begin{corollary}\label{cor:bEOparity}
For any $n\ge 0$,
\begin{align}\label{cong:bEOparity}
\overline{\mathcal{EO}}(n) \equiv \begin{cases}
1 \pmod 2, & \text{if $n=4k(3k-1)$}~\textrm{for~some~}$k$,\\
0 \pmod 2, & \text{otherwise.}
\end{cases}
\end{align}
\end{corollary}

\begin{proof}
Let $\overline{\mathcal{EO}}_c(n)$ denote the number of self-conjugate partitions in
$\overline{\mathfrak{EO}}_n$, then by \eqref{cong:sc-parity} we have $\overline{\mathcal{EO}}(n)\equiv\overline{\mathcal{EO}}_c(n)\pmod2$.

Now each self-conjugate partition $\lambda\in\overline{\mathfrak{EO}}_n$ can be decomposed as
$\lambda=(2m;\mu,\mu)$, where $2m$ is the side of $\lambda$'s Durfee square, and $\mu$ is the subpartition
below the Durfee square, which has an odd number of largest part $2m$, with the remaining parts all being
even and all occurring an even number of times. Together with another copy of $\mu$ to the right of the
Durfee square, they are generated by $(q^{2m}\cdot q^{2m})/(q^8;q^8)_{m}$. This amounts to give the
generating function
\begin{align}
\sum_{n=0}^{\infty}\overline{\mathcal{EO}}_c(n)q^n
 &=\sum_{m=0}^{\infty}\frac{q^{(2m)^2+4m}}{(q^8;q^8)_m}\notag\\
 &=(-q^8;q^8)_{\infty}\label{q-bin-iden}\\
 &\equiv (q^8;q^8)_{\infty}\pmod 2\notag\\
 &=\sum_{m=-\infty}^{\infty}(-1)^m q^{4m(3m-1)},\label{PNT}
\end{align}
where \eqref{q-bin-iden} and \eqref{PNT} follow from Euler's identity \cite[p.~19, Eq.~(2.2.6)]{and} and
Euler's pentagonal number theorem \cite[p.~11, Corollary 1.7]{and}, respectively. Now \eqref{cong:bEOparity}
follows by comparing the coefficients.
\end{proof}

\begin{remark}
Two remarks on Corollary \ref{cor:bEOparity} are in order. For one thing, in a recent paper of Ray and
Barman \cite{RB}, using the theory of modular forms, they derived, among some other things, an infinite
family of congruences modulo $2$ for $\overline{\mathcal{EO}}(n)$ (see \cite[Theorem 1.1]{RB}), and a result
concerning the parity of $\overline{\mathcal{EO}}(n)$ in any arithmetic progression \cite[Theorem 1.4]{RB}.
For another, the identities \eqref{q-bin-iden} and \eqref{PNT} both possess classical combinatorial proofs.
From this perspective, our proof is a combinatorial proof.
\end{remark}

\section{Stable sets in \texorpdfstring{$\mathfrak{EO}$}{EO}}\label{sec:EO}

Equipped with the combinatorial insights from section~\ref{sec:pre}, we investigate in this section various
stable subsets contained in $\fEO$.

\subsection{Further results for \texorpdfstring{$\bfEO$}{overlined-EO}}
We begin by answering a problem raised by Andrews in \cite{and2}:

``Problem 2. Prove Proposition 3.1 combinatorially. (Hopefully more directly than invoking [1].)''

The Proposition 3.1 referred to above is the following generating function of $\overline{\mathcal{EO}}(m,n)$,
the number of partitions in $\overline{\mathfrak{EO}}_n$ whose even-odd crank equals $m$.
\begin{align}\label{gf:bEO(m,n)}
\sum_{n=0}^{\infty}\sum_{m=-\infty}^{\infty}\bEO(m,n)z^m q^n
=\dfrac{(q^4;q^4)_\infty}{(z^2q^2;q^4)_\infty(q^2/z^2;q^4)_\infty}.
\end{align}

We construct an involution on a certain set of $2$-colored partition pairs to establish Theorem
\ref{thm:2-color}, which is a bivariate generalization of \eqref{gf:bEO(m,n)}. This approach is reminiscent
of Franklin's proof of Euler's pentagonal number theorem. We need to make some further definitions.

We consider {\em $2$-colored partitions}, wherein each part of size $k$ can be colored as either $b$ or $ab$
(when $k\ge 2$), such that $ab$-colored parts are all distinct. We remark that this notion of $2$-colored
partition is essentially Corteel and Lovejoy's \textit{overpartition} \cite{CL}, probably appended with parts
of zero (nonoverlined), but we introduce it this way for our convenience.

Using its Ferrers graph, we realize every $2$-colored partition by assigning the neutral weight $q$, and
color weights $b$ or $ab$, such that the leftmost cell of each part is always filled with $b$, while the
rightmost cell of each $ab$-colored part is filled with $a$, with the remaining cells all filled with $q$. We
call such cell-labelled Ferrers graph a \textit{Ferrers diagram}. The weight for each Ferrers diagram $\mu$
is the product over all the weights of its cells, and is denoted as $w(\mu)$; see Fig.~\ref{2-color} for an
example. Note that when there are two parts with the same size but different color, we always put the
$ab$-colored one below the $b$-colored one. Consequently, each cell filled with $a$ must be an outer corner
cell. Denote the set of all Ferrers diagrams as $\mathfrak{F}(a,b;q)$. From now on, we will speak of Ferrers
diagram and $2$-colored partition interchangeably. It is now routine to find its generating function.

\begin{proposition}\label{prop:gf-Ferdia}
The weight generating function of Ferrers diagrams is given by
\begin{align*}
\sum_{\mu\in\mathfrak{F}(a,b;q)}w(\mu) = \frac{(-ab;q)_{\infty}}{(b;q)_{\infty}}.
\end{align*}
\end{proposition}

\begin{figure}[t!]
\begin{ferrers}
\addcellrows{7+6+6+3+3+3+1}
\highlightcellbyletter{1}{1}{$b$}
\highlightcellbyletter{2}{1}{$b$}
\highlightcellbyletter{3}{1}{$b$}
\highlightcellbyletter{4}{1}{$b$}
\highlightcellbyletter{5}{1}{$b$}
\highlightcellbyletter{6}{1}{$b$}
\highlightcellbyletter{7}{1}{$b$}
\highlightcellbyletter{1}{7}{$a$}
\highlightcellbyletter{6}{3}{$a$}
\highlightcellbyletter{1}{2}{$q$}
\highlightcellbyletter{1}{3}{$q$}
\highlightcellbyletter{1}{4}{$q$}
\highlightcellbyletter{1}{5}{$q$}
\highlightcellbyletter{1}{6}{$q$}
\highlightcellbyletter{2}{2}{$q$}
\highlightcellbyletter{2}{3}{$q$}
\highlightcellbyletter{2}{4}{$q$}
\highlightcellbyletter{2}{5}{$q$}
\highlightcellbyletter{2}{6}{$q$}
\highlightcellbyletter{3}{2}{$q$}
\highlightcellbyletter{3}{3}{$q$}
\highlightcellbyletter{3}{4}{$q$}
\highlightcellbyletter{3}{5}{$q$}
\highlightcellbyletter{3}{6}{$q$}
\highlightcellbyletter{4}{2}{$q$}
\highlightcellbyletter{4}{3}{$q$}
\highlightcellbyletter{5}{2}{$q$}
\highlightcellbyletter{5}{3}{$q$}
\highlightcellbyletter{6}{2}{$q$}
\end{ferrers}
\caption{Ferrers diagram for $\mu$ with $w(\mu)=a^2b^7q^{20}$}
\label{2-color}
\end{figure}

Clearly $\mathfrak{F}(0,q;q)$ is in bijection with the set of ordinary partitions. Now define a peculiar set
of $2$-colored partition pairs
\begin{align*}
\fF(a,b,c,d;q):=\{(\lambda,\mu)\in\fF(a,b;q)\times\fF(c,d;q):
\textrm{$\lambda=\epsilon$~or~the~smallest~part~of~$\lambda$}>\#(\mu)\}.
\end{align*}

\begin{lemma}\label{lem:Ferrersdiag and bEO}
There exists a bijection
\begin{align*}
\varphi: \fF(0,q^2/z^2,0,z^2q^2;q^4)\to \bfEO,
\end{align*}
such that a pair of Ferrers diagrams $(\lambda,\mu)\in \fF(0,q^2/z^2,0,z^2q^2;q^4)$ weighted by
$w(\lambda)w(\mu)=z^mq^n$ corresponds to a partition counted by $\overline{\mathcal{EO}}(m,n)$.
\end{lemma}

\begin{proof}
Let $(\lambda,\mu)\in\fF(0,q^2/z^2,0,z^2q^2;q^4)$ with $\#(\lambda)=s$ and $\#(\mu)=t$, then $\lambda_s>t$
from the definition. Moreover, $\lambda\in\fF(0,q^2/z^2;q^4)$ implies that $\lambda$ corresponds to a
partition $\hat{\lambda}\vdash (4|\lambda|-2s)$ into $2s$ odd parts, wherein each part appears an even number
of times and each occurrence is weighted by $z^{-1}$. Thus $\hat{\lambda}$ is weighted by $z^{-2s}$. On the
other hand, the conjugate partition of $\mu\in\fF(0,z^2q^2;q^4)$ is seen to be corresponding to a partition
$\hat{\mu}\vdash (4|\mu|-2t)$ into even parts, with the largest part $2t$ occurring an odd number of times
while all remaining parts occurring an even number of times. $\hat{\mu}$ is weighted by $z^{2t}$.

Now, let $\pi$ be the unique partition obtained from appending $\hat{\mu}$ to $\hat{\lambda}$. We see that
the condition $\lambda_s>t$ ensures that $\pi\in\bfEO$. In addition, we have indeed $\eoc(\pi)=2t-2s$.
Conversely, it should be clear how to start with a partition in $\bfEO$ and recover a unique pair in
$\fF(0,q^2/z^2,0,z^2q^2;q^4)$. See Fig.~\ref{varphi} below for a concrete example.
\end{proof}

\begin{figure}[t!]
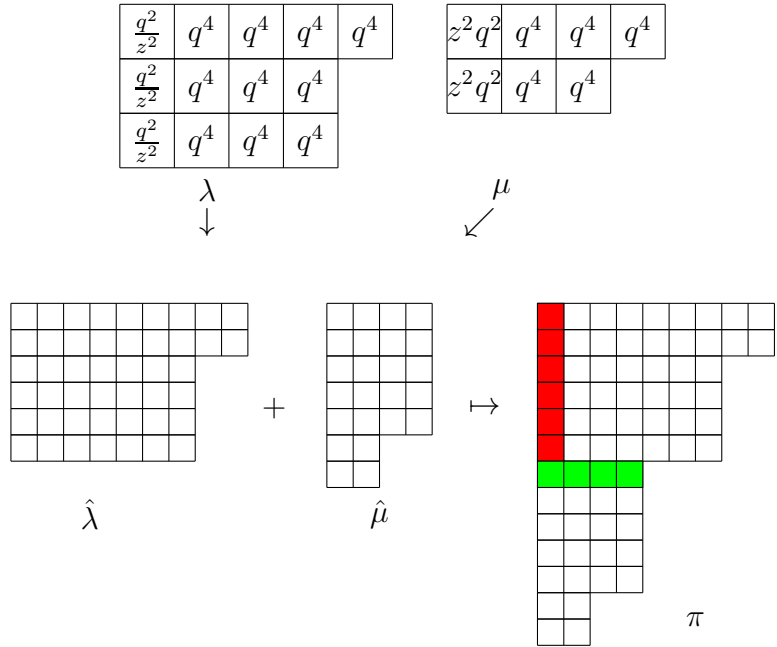

\begin{ferrers}[1.45]
\addcellrows{5+4+4}
\highlightcellbyletter{1}{1}{$\frac{q^2}{z^2}$}
\highlightcellbyletter{2}{1}{$\frac{q^2}{z^2}$}
\highlightcellbyletter{3}{1}{$\frac{q^2}{z^2}$}
\highlightcellbyletter{1}{2}{$q^4$}
\highlightcellbyletter{1}{3}{$q^4$}
\highlightcellbyletter{1}{4}{$q^4$}
\highlightcellbyletter{1}{5}{$q^4$}
\highlightcellbyletter{2}{2}{$q^4$}
\highlightcellbyletter{2}{3}{$q^4$}
\highlightcellbyletter{2}{4}{$q^4$}
\highlightcellbyletter{3}{2}{$q^4$}
\highlightcellbyletter{3}{3}{$q^4$}
\highlightcellbyletter{3}{4}{$q^4$}
\setbase{3}{0}
\addcellrows{4+3}
\highlightcellbyletter{1}{1}{$z^2q^2$}
\highlightcellbyletter{2}{1}{$z^2q^2$}
\highlightcellbyletter{1}{2}{$q^4$}
\highlightcellbyletter{1}{3}{$q^4$}
\highlightcellbyletter{1}{4}{$q^4$}
\highlightcellbyletter{2}{3}{$q^4$}
\highlightcellbyletter{2}{2}{$q^4$}
\addtext{-2.2}{-1.7}{$\lambda$}
\addtext{-2.2}{-2}{$\downarrow$}
\addtext{0.5}{-1.7}{$\mu$}
\addtext{0.3}{-2}{$\swarrow$}
\end{ferrers}
\begin{ferrers}[0.7]
\setbase{0}{-5}
\addcellrows{9+9+7+7+7+7}
\addtext{1.5}{-4}{$\hat{\lambda}$}
\addtext{5}{-2}{$+$}
\putright
\addcellrows{4+4+4+4+4+2+2}
\addtext{1}{-4}{$\hat{\mu}$}
\addtext{3}{-2}{$\mapsto$}
\putright
\addcellrows{9+9+7+7+7+7+4+4+4+4+4+2+2}
\highlightcellbycolor{1}{1}{red}
\highlightcellbycolor{2}{1}{red}
\highlightcellbycolor{3}{1}{red}
\highlightcellbycolor{4}{1}{red}
\highlightcellbycolor{5}{1}{red}
\highlightcellbycolor{6}{1}{red}
\highlightcellbycolor{7}{1}{green}
\highlightcellbycolor{7}{2}{green}
\highlightcellbycolor{7}{3}{green}
\highlightcellbycolor{7}{4}{green}
\addtext{3}{-6}{$\pi$}
\end{ferrers}
\caption{The correspondence $\varphi: (\lambda,\mu)\to \pi$ with $\eoc(\pi)=4-6=-2$}
\label{varphi}
\end{figure}

In view of Lemma~\ref{lem:Ferrersdiag and bEO}, the following theorem indeed covers Andrews's formula
\eqref{gf:bEO(m,n)} as the special case of setting $a\to q^2/z^2$, $b\to z^2q^2$, $q\to q^4$.

\begin{theorem}\label{thm:2-color}
We have
\begin{align}\label{eq:ab-involution}
\sum_{(\lambda,\mu)\in\mathfrak{F}(0,a;q)\times\mathfrak{F}(-a,b;q)}w(\lambda)w(\mu)
=\dfrac{(ab;q)_{\infty}}{(a;q)_{\infty}(b;q)_{\infty}}
=\sum_{(\lambda,\mu)\in\fF(0,a,0,b;q)}w(\lambda)w(\mu).
\end{align}
\end{theorem}
\begin{proof}
The first equality follows directly from applying Proposition~\ref{prop:gf-Ferdia} for the pair
$(\lambda,\mu)$. To show the weight equivalence between $\fF(0,a;q)\times\fF(-a,b;q)$ and $\fF(0,a,0,b;q)$,
it suffices to construct a weight-preserving, sign-reversing involution\footnote{That is, we have $w(\lambda^*)w(\mu^*)=-w(\lambda)w(\mu)$ whenever $(\lambda^*,\mu^*)\neq (\lambda,\mu)$.}, say
$*:(\lambda,\mu)\mapsto (\lambda^*,\mu^*)$, over $\fF(0,a;q)\times\fF(-a,b;q)$, with $\fF(0,a,0,b;q)$ being
precisely the set of fixed pairs; see Fig.~\ref{fig:involution} for an example of the map $*$, where we use
$\bar{a}$ in place of the weight $-a$ to save space.

Given a pair $(\lambda,\mu)$ with $\#(\mu)=t$, let $\lambda_i$ be the largest part of $\lambda$ that is no
greater than $t$, note that the leftmost cell of $\lambda_i$ is filled with $a$ since $\lambda\in\fF(0,a;q)$.
Also let $j$ be the largest index of an $ab$-colored part of $\mu$, note that the rightmost cell of $\mu_j$
is filled with $-a$ since $\mu\in\fF(-a,b;q)$. If such a $\lambda_i$ (resp.~$j$) does not exist, we simply
set $\lambda_i=0$ (resp.~$j=0$). Now we compare $\lambda_i$ with $j$, and consider the following three
cases.
\begin{enumerate}
\item[Case I] $\lambda_i\le j\neq 0$. Insert $j$ as a new part into $\lambda$ to get $\lambda^*$, and
convert $\mu$ to $\mu^*$ whose parts are given by
\begin{align*}
\mu^*_k=\begin{cases}
\mu_k-1, &\textrm{if}~1\le k\le j,\\
\mu_k, &\textrm{otherwise.}\\
\end{cases}
\end{align*}
Note that $\mu_j$ is $ab$-colored hence $\mu_j>\mu_{j+1}$, which ensures $\mu^*_j\ge\mu^*_{j+1}$, and thus
$\mu^*$ is well-defined.

\item[Case II] $\lambda_i>j$. Remove part $\lambda_i$ from $\lambda$ to get $\lambda^*$. Convert $\mu$ to
$\mu^*$ whose parts are given by
\begin{align*}
\mu^*_k=\begin{cases}
\mu_k+1, &\textrm{if}~1\le k\le \lambda_i,\\
\mu_k, &\textrm{otherwise.}\\
\end{cases}
\end{align*}
We fill the rightmost cell of $\mu^*_{\lambda_i}$ with $-a$, while other inserted cells are filled with $q$.
Note that originally $\mu_{\lambda_i}$ cannot be $ab$-colored, or we are not in Case II. Therefore color
part $\mu^*_{\lambda_i}$ as $ab$ is legitimate.

\item[Case III] $\lambda_i=j=0$. We stay put and take $(\lambda^*,\mu^*)=(\lambda,\mu)$. Note that we are in
this case if and only if all parts of $\lambda$ are greater than $t$ (or $\lambda=\epsilon$), and none part
of $\mu$ is $ab$-colored. In other words, $(\lambda,\mu)\in\fF(0,a,0,b;q)$, as desired.
\end{enumerate}
It should be clear that the operations conducted in Cases I and II are inverses of each other, and
$*:(\lambda,\mu)\mapsto(\lambda^*,\mu^*)$ as described above is an involution over
$\fF(0,a;q)\times\fF(-a,b;q)$, such that $w(\lambda^*)w(\mu^*)=-w(\lambda)w(\mu)$ whenever
$(\lambda^*,\mu^*)\neq(\lambda,\mu)$. And the set of pairs fixed by $*$ is exactly $\fF(0,a,0,b;q)$. We have
now proven \eqref{eq:ab-involution}.
\end{proof}

\begin{figure}[htp]
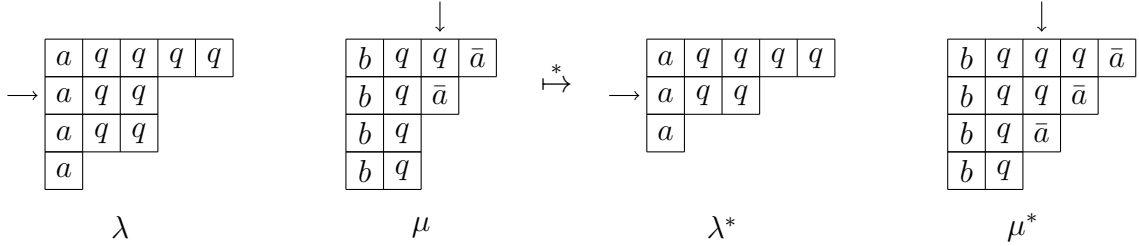

\begin{ferrers}
\addcellrows{5+3+3+1}
\highlightcellbyletter{1}{1}{$a$}
\highlightcellbyletter{2}{1}{$a$}
\highlightcellbyletter{3}{1}{$a$}
\highlightcellbyletter{4}{1}{$a$}
\highlightcellbyletter{1}{2}{$q$}
\highlightcellbyletter{1}{3}{$q$}
\highlightcellbyletter{1}{4}{$q$}
\highlightcellbyletter{1}{5}{$q$}
\highlightcellbyletter{2}{2}{$q$}
\highlightcellbyletter{2}{3}{$q$}
\highlightcellbyletter{3}{2}{$q$}
\highlightcellbyletter{3}{3}{$q$}
\highlightrowbyarrow{2}
\addtext{1}{-2.5}{$\lambda$}
\setbase{4}{0}
\addcellrows{4+3+2+2}
\highlightcellbyletter{1}{1}{$b$}
\highlightcellbyletter{2}{1}{$b$}
\highlightcellbyletter{3}{1}{$b$}
\highlightcellbyletter{4}{1}{$b$}
\highlightcellbyletter{1}{4}{$\bar a$}
\highlightcellbyletter{2}{3}{$\bar a$}
\highlightcellbyletter{1}{2}{$q$}
\highlightcellbyletter{1}{3}{$q$}
\highlightcellbyletter{2}{2}{$q$}
\highlightcellbyletter{3}{2}{$q$}
\highlightcellbyletter{4}{2}{$q$}
\highlightcolumnbyarrow{3}
\addtext{1}{-2.5}{$\mu$}
\addtext{2.8}{-0.5}{$\stackrel{*}{\mapsto}$}

\setbase{8}{0}
\addcellrows{5+3+1}
\highlightcellbyletter{1}{1}{$a$}
\highlightcellbyletter{2}{1}{$a$}
\highlightcellbyletter{3}{1}{$a$}
\highlightcellbyletter{1}{2}{$q$}
\highlightcellbyletter{1}{3}{$q$}
\highlightcellbyletter{1}{4}{$q$}
\highlightcellbyletter{1}{5}{$q$}
\highlightcellbyletter{2}{2}{$q$}
\highlightcellbyletter{2}{3}{$q$}
\highlightrowbyarrow{2}
\addtext{1}{-2.5}{$\lambda^*$}
\setbase{12}{0}
\addcellrows{5+4+3+2}
\highlightcellbyletter{1}{1}{$b$}
\highlightcellbyletter{2}{1}{$b$}
\highlightcellbyletter{3}{1}{$b$}
\highlightcellbyletter{4}{1}{$b$}
\highlightcellbyletter{1}{5}{$\bar a$}
\highlightcellbyletter{2}{4}{$\bar a$}
\highlightcellbyletter{3}{3}{$\bar a$}
\highlightcellbyletter{1}{2}{$q$}
\highlightcellbyletter{1}{3}{$q$}
\highlightcellbyletter{1}{4}{$q$}
\highlightcellbyletter{2}{2}{$q$}
\highlightcellbyletter{2}{3}{$q$}
\highlightcellbyletter{3}{2}{$q$}
\highlightcellbyletter{4}{2}{$q$}
\highlightcolumnbyarrow{3}
\addtext{1}{-2.5}{$\mu^*$}
\end{ferrers}
\caption{An example of the involution $*$ (Case II with $\lambda_i=3>j=2$)}\label{fig:involution}
\end{figure}

Next, we proceed to establish Theorem~\ref{THM:AB-cong}. The first step is to note the vanishing of certain
$N_{\textrm{eo}}(k,m,n)$, which can be readily deduced from Lemma~\ref{lem:Ferrersdiag and bEO}.

\begin{corollary}\label{coro:eoc-refine}
For any $n\geq0$,
\begin{align}
N_{\emph{eo}}(0,4,4n+2)=N_{\emph{eo}}(2,4,4n+4)=0.\label{zero-iden}
\end{align}
\end{corollary}

\begin{proof}
As explained in the proof of Lemma~\ref{lem:Ferrersdiag and bEO}, for each $\pi=\varphi(\lambda,\mu)$ with
$\#(\lambda)=s$ and $\#(\mu)=t$, we have
\begin{align*}
\textrm{eoc}(\pi)=2t-2s\equiv 2t+2s\equiv |\pi|\pmod 4,
\end{align*}
which is clearly equivalent to \eqref{zero-iden}.
\end{proof}

\begin{proof}[Proof of Theorem \ref{THM:AB-cong}]
We first show the divisibility by $5$. Note that when $r\not\equiv 0 \pmod k$,
$N_{\textrm{eo}}(r,k,n)=N_{\textrm{eo}}(k-r,k,n)$ via the map of conjugation, hence for $0<r<k/2$ we have
\begin{align*}
r\big(N_{\textrm{eo}}(r,k,n)-N_{\textrm{eo}}(k-r,k,n)\big)
 &=r\sum_{\substack{\lambda\in\overline{\mathfrak{EO}}_n \\ \eoc(\lambda)\equiv r\pmod{k}}}
\big(\mathcal{O}(\lambda)-\mathcal{O}(\lambda')\big) \\
 &\equiv-r^2N_{\textrm{eo}}(r,k,n)\pmod k, \qquad (\text{by~\eqref{eq:Andrews=Stanley}})
\end{align*}
In particular, if $\gcd(r,k)=1$,
\begin{align*}
N_{\textrm{eo}}(r,k,n)-N_{\textrm{eo}}(k-r,k,n)\equiv-rN_{\textrm{eo}}(r,k,n)\pmod{k}.
\end{align*}
Therefore,
\begin{align*}
\big(N_{\textrm{eo}}(1,5,n)-N_{\textrm{eo}}(4,5,n)\big)
 &+2\big(N_{\textrm{eo}}(2,5,n)-N_{\textrm{eo}}(3,5,n)\big)\\
 &\quad\equiv-N_{\textrm{eo}}(1,5,n)-4N_{\textrm{eo}}(2,5,n)\pmod{5}.
\end{align*}
Finally, according to \cite[Theorem 2.1.5]{Pas},
\begin{align*}
N_{\textrm{eo}}(1,5,10n)=N_{\textrm{eo}}(2,5,10n),
\end{align*}
and \eqref{eq:1/5} already gives us
\begin{align*}
N_{\textrm{eo}}(1,5,10n+8)=N_{\textrm{eo}}(2,5,10n+8).
\end{align*}
So both congruences \eqref{AB-cong:10n8} and \eqref{AB-cong:10n} hold modulo $5$. Since each partition
$\lambda\in\overline{\mathfrak{EO}}$ has an even number of odd parts, we get the modulo $2$ results
immediately. All it remains, is to show that
\begin{align}
 &NT_{\textrm{eo}}(1,5,10n+8)-NT_{\textrm{eo}}(4,5,10n+8)\nonumber\\
 &\quad=\sum_{\substack{\lambda\in\bfEO_{10n+8} \\ \eoc(\lambda)\equiv 1\pmod{5}}}
\big(\mathcal{O}(\lambda)-\mathcal{O}(\lambda')\big)
=-\sum_{\substack{\lambda\in\bfEO_{10n+8} \\ \eoc(\lambda)\equiv1\pmod{5}}}
\eoc(\lambda)\equiv0\pmod 4.\label{cong:mod4}
\end{align}
To this end, we write $n=4m+i$ with $i=0,1,2,3$, and discuss the following two cases.
\begin{enumerate}[(1)]
\item The case $i\equiv0\pmod{2}$. In this case $10n+8\equiv 0 \pmod 4$ so each
$\lambda\in\overline{\mathfrak{EO}}_{10n+8}$ satisfies $\textrm{eoc}(\lambda)\equiv 0\pmod 4$ by
\eqref{zero-iden}, thus \eqref{cong:mod4} holds.

\item The case $i\equiv1\pmod{2}$. In this case $10n+8\equiv 2 \pmod 4$ so each
$\lambda\in\overline{\mathfrak{EO}}_{10n+8}$ satisfies $\textrm{eoc}(\lambda)\equiv 2\pmod 4$ by
\eqref{zero-iden}, and the total number of summands is
\begin{align*}
N_{\textrm{eo}}(1,5,40m+10i+8)=\frac{1}{5}\overline{\mathcal{EO}}(40m+10i+8)\equiv0\pmod{2},
\quad\text{(by~\eqref{cong:bEOparity})}
\end{align*}
i.e., it is an even number. So \eqref{cong:mod4} still holds.
\end{enumerate}
This completes the proof.
\end{proof}

\subsection{Three more stable subsets of \texorpdfstring{$\fEO$}{EO}}\label{subsec:subsets of EO}
With the discussion of $\overline{\fEO}$ after Definition \ref{def:stable} in mind, other stable subsets of
$\fEO$ besides $\bfEO$ naturally present themselves. We study three of them in this subsection. First note
that $\mathfrak{EO}$ itself is not stable, as can be seen from Fig. \ref{ferrers}, wherein
$\lambda'\in\mathfrak{EO}$ but $\lambda\not\in\mathfrak{EO}$. The three stable subsets of $\mathfrak{EO}$
are defined as follows, the first of which is the largest one, in the sense that there exists no stable
subset $\mathfrak{C}$, such that $\mathfrak{EO}^{(1)}\subsetneq\mathfrak{C}\subsetneq\mathfrak{EO}$.
\begin{align*}
\mathfrak{EO}^{(1)} &:=\{\lambda\in\mathfrak{EO}\colon\textrm{at~most~one~part~of~$\lambda$~aont}\},\\
\mathfrak{EO}^{(2)} &:=\{\lambda\in\mathfrak{EO}\colon
\textrm{only~the~smallest~odd~part~of~$\lambda$~aont}\},\\
\mathfrak{EO}^{(3)} &:=\{\lambda\in\mathfrak{EO}\colon
\textrm{all~parts~of~$\lambda$~odd~and~only~the~largest~odd~part~of~$\lambda$ aont}\}.
\end{align*}
Here ``aont'' stands for ``{\bf a}ppears an {\bf o}dd {\bf n}umber of {\bf t}imes''. Note that
$\epsilon\not\in\fEO^{(2)}\cup\fEO^{(3)}$. Interpreting the defining conditions for $\mathfrak{EO}^{(1)}$,
$\mathfrak{EO}^{(2)}$ and $\mathfrak{EO}^{(3)}$ in terms of the profile words for the respective
partitions, we obtain the following connection with mock theta functions $\nu(q)$ and $\omega(q)$.

\begin{theorem}\label{thm:mock-theta}
The three subsets $\fEO^{(1)}$, $\fEO^{(2)}$ and $\fEO^{(3)}$ are all stable. Furthermore, we have
\begin{align}
\sum_{n=1}^\infty\mathcal{EO}^{(2)}(n)q^n &=\dfrac{1}{2}q\big(\nu(q)+\nu(-q)\big),\label{gf2-mock}\\
\sum_{n=1}^\infty\mathcal{EO}^{(3)}(n)q^n &=q\!\:\omega(q^2),\label{gf3-mock}
\end{align}
where $\nu(q)$ and $\omega(q)$ are defined as in \eqref{def:nu} and \eqref{def:omega}.
\end{theorem}

\begin{proof}
Given a profile word $w_{\lambda}=(e_1,n_1,\ldots,e_k,n_k)$ associated with partition $\lambda$. It can be
characterized respectively as having at most one odd number among all $n_i$'s when $\lambda\in\fEO^{(1)}$,
having a consecutive odd pair $(e_i,n_i)$ when $\lambda\in\fEO^{(2)}$, and having only $e_1$ and $n_k$ as
odd components when $\lambda\in\fEO^{(3)}$. Each of the above characterizations is clearly invariant under
the reversal $w_{\lambda}\to w_{\lambda'}$, hence all three subsets $\mathfrak{EO}^{(1)}$,
$\mathfrak{EO}^{(2)}$ and $\mathfrak{EO}^{(3)}$ are stable.

To obtain \eqref{gf2-mock}, we observe that there is a natural bijection between the set
$\mathfrak{EO}^{(2)}_n$ and $\overline{\mathfrak{EO}}_{n-1}$. For a given partition
$\lambda\in\mathfrak{EO}^{(2)}_n$ with $n\ge 1$, subtract one from the last smallest odd part of $\lambda$
and denote this new partition by $\mu$. Obviously, $\mu\in\overline{\mathfrak{EO}}_{n-1}$. Conversely, for a
given partition $\mu\in\overline{\mathfrak{EO}}_{n-1}$, adding one to its first largest even
part\footnote{If $\mu$ has none even part, simply append $1$ as a new part.}
recovers for us a partition $\lambda\in\fEO^{(2)}_n$. Therefore,
\begin{align*}
\sum_{n=1}^\infty\mathcal{EO}^{(2)}(n)q^n=\sum_{n=1}^\infty\overline{\mathcal{EO}}(n-1)q^{n}
=q\dfrac{(q^4;q^4)_\infty^3}{(q^2;q^2)_\infty^2}=\dfrac{1}{2}q\big(\nu(q)+\nu(-q)\big),
\end{align*}
where the second and last equalities follow from Corollary 3.2 and Theorem 1.1 in \cite{and2}.

To handle the connection with $\omega(q)$ and prove \eqref{gf3-mock}, we recall a variation of the standard
Ferrers graph introduced by Andrews in \cite{and5}, called the ``odd Ferrers graph''. It consists of a Ferrers
graph using $2$'s with a surrounding border of $1$'s; see Fig.~\ref{fig:2-oddFer} for an illustration along
with its ``$2$-dilation''. As noted in \cite[p.~60]{and5}, the generating function for odd Ferrers graphs is
$q\!\:\omega(q)$. Now we immediately have \eqref{gf3-mock} by noticing that each partition in $\fEO^{(3)}$ can
be uniquely represented by a $2$-dilated odd Ferrers graph with its top-left corner cell subtracted by $1$.
\end{proof}

\begin{figure}[tbh]
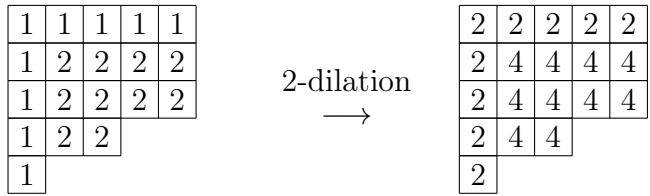

\begin{ferrers}
\addcellrows{5+5+5+3+1}
\highlightcellbyletter{1}{1}{$1$}
\highlightcellbyletter{2}{1}{$1$}
\highlightcellbyletter{3}{1}{$1$}
\highlightcellbyletter{4}{1}{$1$}
\highlightcellbyletter{5}{1}{$1$}
\highlightcellbyletter{1}{2}{$1$}
\highlightcellbyletter{1}{3}{$1$}
\highlightcellbyletter{1}{4}{$1$}
\highlightcellbyletter{1}{5}{$1$}
\highlightcellbyletter{2}{2}{$2$}
\highlightcellbyletter{2}{3}{$2$}
\highlightcellbyletter{2}{4}{$2$}
\highlightcellbyletter{2}{5}{$2$}
\highlightcellbyletter{3}{2}{$2$}
\highlightcellbyletter{3}{3}{$2$}
\highlightcellbyletter{3}{4}{$2$}
\highlightcellbyletter{3}{5}{$2$}
\highlightcellbyletter{4}{2}{$2$}
\highlightcellbyletter{4}{3}{$2$}
\setbase{6}{0}
\addcellrows{5+5+5+3+1}
\highlightcellbyletter{1}{1}{$2$}
\highlightcellbyletter{2}{1}{$2$}
\highlightcellbyletter{3}{1}{$2$}
\highlightcellbyletter{4}{1}{$2$}
\highlightcellbyletter{5}{1}{$2$}
\highlightcellbyletter{1}{2}{$2$}
\highlightcellbyletter{1}{3}{$2$}
\highlightcellbyletter{1}{4}{$2$}
\highlightcellbyletter{1}{5}{$2$}
\highlightcellbyletter{2}{2}{$4$}
\highlightcellbyletter{2}{3}{$4$}
\highlightcellbyletter{2}{4}{$4$}
\highlightcellbyletter{2}{5}{$4$}
\highlightcellbyletter{3}{2}{$4$}
\highlightcellbyletter{3}{3}{$4$}
\highlightcellbyletter{3}{4}{$4$}
\highlightcellbyletter{3}{5}{$4$}
\highlightcellbyletter{4}{2}{$4$}
\highlightcellbyletter{4}{3}{$4$}
\addtext{-1.5}{-1}{$2$-dilation}
\addtext{-1.5}{-1.5}{$\longrightarrow$}
\end{ferrers}
\caption{An odd Ferrers graph and its $2$-dilation}
\label{fig:2-oddFer}
\end{figure}

\begin{remark}
Several remarks on Theorem~\ref{thm:mock-theta} are necessary. Firstly, the coefficients of $\nu(-q)$ and
$\omega(q)$ also possess their own partition-theoretic interpretations, so it is interesting to give a bijective
proof of \eqref{gf2-mock} and \eqref{gf3-mock}. Secondly, there are various investigations on arithmetic
properties of coefficients of $\nu(q)$ and $\omega(q)$ (see, for example, \cite{GP, BO, APSY}). This suggests
that there should be various congruences for functions $\EO^{(2)}(n)$ and $\EO^{(3)}(n)$, inherited from those
of $\nu(q)$ and $\omega(q)$. Finally, it is not easy to find the generating function for $\mathcal{EO}^{(1)}(n)$
following a similar line of proving \eqref{gf2-mock} and \eqref{gf3-mock}.
\end{remark}

Let $\nu(-q):=\sum_{n\ge 0}p_{\nu}(n)q^n$, then it was shown by Andrews, Dixit, and Yee~\cite[Theorem.~4.1]{ADY}
that $p_{\nu}(n)$ counts the number of distinct partitions of $n$ in which all odd parts are less than twice the
smallest part.\footnote{Actually, when such a nonempty partition contains no odd parts, it will be counted twice
in $p_{\nu}(n)$. For instance, $p_{\nu}(6)=4$, counting both partitions $(6)$ and $(4,2)$ twice. This correction
has been confirmed with Yee \cite{Yeep}.} Thanks to the new connections made in Theorem~\ref{thm:mock-theta}, we
are rewarded with an alternative interpretation of $p_{\nu}(n)$. Note that the even $n$ case of the following
result is essentially rephrasing Theorem 1.1 of \cite{and2}. From this perspective, the next corollary is a
companion to Theorem 1.1 of \cite{and2}.

\begin{corollary}
For any $n\ge 0$, $p_\nu(n)$ also counts the number of partitions of $n$ in $\mathfrak{EO}$, such that when $n$
is even, then only the largest even part appears an odd number of times; when $n$ is odd, then all parts are odd
and only the largest odd part appears an odd number of times.
\end{corollary}

\begin{proof}
According to \cite[p. 72]{Wat}, we find that
\begin{align}\label{id1-nu-omega}
\nu(q)+q\!\:\omega(q^2)=\dfrac{(q^4;q^4)_\infty^3}{(q^2;q^2)_\infty^2}.
\end{align}
Replacing $q$ by $-q$ yields
\begin{align}\label{id2-nu-omega}
\nu(-q)-q\!\:\omega(q^2)=\dfrac{(q^4;q^4)_\infty^3}{(q^2;q^2)_\infty^2}.
\end{align}
In view of \eqref{gf2-mock}--\eqref{id2-nu-omega},
\begin{align*}
\sum_{n=1}^\infty\mathcal{EO}^{(2)}(n)q^n+\sum_{n=1}^\infty\mathcal{EO}^{(3)}(n)q^{n+1}
=q\nu(-q)=\sum_{n=0}^\infty p_\nu(n)q^{n+1},
\end{align*}
from which we obtain
\begin{align*}
\mathcal{EO}^{(2)}(n+1)+\mathcal{EO}^{(3)}(n)=p_\nu(n),\quad \text{for }n\ge 0.
\end{align*}
Finally, we observe that the powers of $q$ in both $q\big(\nu(q)+\nu(-q)\big)/2$ and $q\!\:\omega(q^2)$ are odd,
which results in $\EO^{(2)}(2m)=\EO^{(3)}(2m)=0$ for every integer $m\ge 0$. Now we apply the identity
$\mathcal{EO}^{(2)}(n+1)=\overline{\mathcal{EO}}(n)$ to finish the proof.
\end{proof}

\section{Stable sets in \texorpdfstring{$\mathfrak{OE}$}{OE}}\label{sec:OE}
In the first subsection, we present a proof of Theorem \ref{thm:bOE-srank-gf} and characterize completely the
parity of $\bOE(n)$. We consider further subsets of $\bfOE$ in subsection~\ref{subsec:bfOE-refine}.

\subsection{Results for \texorpdfstring{$\bfOE$}{overlined-OE}}\label{subsec:bfOE}

For a partition $\lambda\in\bfOE$, its profile word $w_{\lambda}$ begins and ends with an odd number $e_1$ and
$n_k$ respectively, and contains precisely one consecutive odd pair $(n_i,e_{i+1})$, with remaining letters all
being even. This characterization is seen to be invariant under conjugation thus $\bfOE$ is stable. Then, with
Observation \ref{ob:key} in mind, we see \eqref{gf:boe}, which we prove next, parallels Andrews' generating
function \eqref{gf:bEO(m,n)} of $\bEO(m,n)$ nicely.

\begin{proof}[Proof of Theorem \ref{thm:bOE-srank-gf}]
Recall that for each $\lambda\in\bfOE$, it contains both even and odd parts, each odd part is smaller than each
even part, and ONLY the largest odd part and the largest even part appears an odd number of times. Suppose
$\lambda$ (resp.~$\lambda'$) has $2n-1$ (resp.~$2m-1$) even parts, for certain $n,m\ge 1$. From these
constraints we can uniquely dissect $\lambda$ into five pieces; see Fig.~\ref{fig:boe-dissect} for an
illustration. The top-left rectangle contains $(2n-1)\times(2m-1)$ cells which contribute $q^{(2m-1)(2n-1)}$
to the generating function; a horizontal strip representing the largest odd part of $\lambda$ with contribution
$zq^{2m-1}$; a vertical strip representing the largest odd part of $\lambda'$ with contribution
$z^{-1}q^{2n-1}$; a subpartition $\alpha$ below the horizontal strip; and a subpartition $\beta$ to the right
of the vertical strip. Note that $\alpha$ is a partition into odd parts each occurring an even number of times,
with the largest part no greater than $2m-1$, and its contribution is $z^{\#(\alpha)}q^{|\alpha|}$. Similarly,
$\beta$'s conjugation $\beta'$ is also a partition into odd parts each occurring an even number of times, with
the largest part no greater than $2n-1$, and its contribution is $z^{-\#(\beta')}q^{|\beta'|}$. Finally, it
suffices to note that all such $\alpha$'s (resp.~$\beta$'s) are generated by $1/(z^2q^2;q^4)_m$
(resp.~$1/(q^2/z^2;q^4)_n$). Putting together all five pieces we arrive at \eqref{gf:boe}.
\end{proof}

\begin{figure}[tbh]
\begin{ferrers}[0.7]
\addsketchrows{12+12+12+10+10+7+5+5+3+3+3+3}
\addline{4}{0}{4}{-2.5}
\addline{3.5}{0}{3.5}{-2.5}
\addline{0}{-2.5}{3.5}{-2.5}
\addline{0}{-3}{2.5}{-3}
\addtext{5}{-1}{$\beta$}
\addtext{1}{-4}{$\alpha$}
\addtext{-1}{-1}{$2n-1$}
\addtext{2}{.6}{$2m-1$}
\end{ferrers}
\caption{Dissection of a partition $\lambda\in\bfOE$}
\label{fig:boe-dissect}
\end{figure}

Next, we study the parity of $\overline{\mathcal{OE}}(n)$, in the hope of getting a complete characterization
analogous to Corollary \ref{cor:bEOparity}.

Another third order mock theta function due to Watson \cite[p. 62]{Wat} is given by
\begin{align}\label{def:psi}
\psi^{(3)}(q)=\sum_{n=1}^\infty\dfrac{q^{n^2}}{(q;q^2)_n}:=\sum_{n=1}^\infty p_\psi(n)q^n.
\end{align}
Utilizing some techniques from analytic and algebraic number theory, Wang \cite[Theorem 4.2]{Wang}
established the following complete characterization modulo 2 for $p_\psi(n)$:
\begin{align}
p_\psi(m)\equiv1\pmod{2}~\Longleftrightarrow~24m-1=p^{4\alpha+1}k^2\;\;
\textrm{for~some~prime~$p$~coprime~to~$k$}.\label{cha-mod-2}
\end{align}

Relying on Wang's result, we are able to fully characterize $\bOE(n)$ modulo $2$.

\begin{theorem}\label{thm:bOEparity}
For any $n\geq1$,
\begin{align}\label{bOEparity}
\bOE(n)\equiv
\begin{cases}
1 \pmod 2, &\textrm{if~$n\equiv3\pmod4$}\;\;\textrm{and}\;\;6n+5=p^{4\alpha+1}k^2\\
  &\textrm{for~some~prime~$p$~coprime~to~$k$},\\
0 \pmod 2, &\textrm{otherwise}.
\end{cases}
\end{align}
\end{theorem}

\begin{proof}
By Proposition~\ref{Prop:conj}, it suffices to consider only self-conjugate partitions in $\bfOE$. Their
generating function is given by
\begin{align*}
\sum_{n=0}^\infty\bOE_c(n)q^n &= q^{-1}\sum_{n=1}^{\infty}\frac{q^{4n^2}}{(q^4;q^8)_n}
=q^{-1}\sum_{m=1}^{\infty}p_{\psi}(m)q^{4m}\qquad\textrm{(by~\eqref{def:psi})}.
\end{align*}
It follows that
\begin{align}
\bOE(n)\equiv
\begin{cases}
p_\psi(m)\pmod2, &\textrm{if}\;\;n=4m-1,\\
\quad0\quad\;\pmod 2, &\textrm{otherwise.}
\end{cases}\label{cong-relate}
\end{align}
Now \eqref{bOEparity} follows from \eqref{cong-relate} and \eqref{cha-mod-2} immediately.
\end{proof}

With the aid of Theorem~\ref{thm:bOEparity}, it is possible to deduce Ramanujan-type congruence for a certain
arithmetic progression. We record one example below.

\begin{corollary}\label{thm:psi-mod2}
For any $n>0$ and $n\not\equiv 0\pmod 5$,
\begin{align}
\bOE(20n-5)&\equiv0\pmod 2.\label{parity cong:bOE}
\end{align}
\end{corollary}
\begin{proof}
Writing $n=5m+i$ for $i=1,2,3,4$, we see that \eqref{parity cong:bOE} is equivalent to
\begin{align*}
\bOE(100m+15)\equiv\bOE(100m+35) &\equiv0\pmod 2,\\
\bOE(100m+55)\equiv\bOE(100m+75) &\equiv0\pmod 2.
\end{align*}
In view of the relation \eqref{cong-relate}, we only need to show that
\begin{align}
p_\psi(25m+4)\equiv p_\psi(25m+9)\equiv p_{\psi}(25m+14)\equiv
p_\psi(25m+19)\equiv0\pmod 2.\label{parity cong:psi}
\end{align}
Here we only prove the congruence $p_{\psi}(25m+4)\equiv0\pmod{2}$, because the proofs of the remaining cases
are quite similar.

For any $m\geq0$, note that $24(25m+4)-1=5(120m+19)$, $5\nmid(120m+19)$ and $120m+19$ can not be a square (a
square is congruent to $0$ or $1$ modulo $4$). According to \eqref{cha-mod-2}, we obtain the desired congruence.
\end{proof}

\begin{remark}
Quite recently, Chen and Garvan \cite[Eq. (4.31)]{CG22} derived an infinite family of congruences modulo $4$
satisfied by $p_\psi(n)$. One easily derives from their results the following strengthening of
\eqref{parity cong:psi}. For any $n\geq0$,
\begin{align*}
p_\psi(25n+9)\equiv p_\psi(25n+14)\equiv0\pmod{4}.
\end{align*}
From this perspective, it might be interesting to pursue mod $4$ results for $\bOE(n)$.
\end{remark}

\subsection{Subsets of \texorpdfstring{$\bfOE$}{overlined-OE}}\label{subsec:bfOE-refine}
Given a partition $\lambda$, let $\cE(\lambda)$ be the number of even parts in $\lambda$. We further refine
$\bfOE$ as follows. Let $k\ge 0$ and
\begin{align*}
\bfOE_{k}:=\{\lambda\in\bfOE: \cE(\lambda)-\cE(\lambda')=2k\}.
\end{align*}
Noting that each partition $\lambda\in\bfOE$ must be a partition of an odd number, we define $p_k(0)=p_k(1)=0$,
and for all $n\ge 2$ let
\begin{align*}
p_k(n)&:=|\{\lambda\in\bfOE_k:|\lambda|=2n-1\}|.
\end{align*}

The following explicit generating functions for $p_0(n)$ and $p_1(n)$ presage their connection with the mock
theta function $\omega(q)$ (see Corollary \ref{Cor:relat:mock-theta}).

\begin{corollary}
There holds
\begin{align}
\sum_{n=0}^\infty p_0(n)q^n &=\sum_{n=1}^{\infty}\frac{q^{2n^2}}{(q;q^2)_{n}^2},\label{gf:p0}\\
\sum_{n=0}^\infty p_1(n)q^n &=\sum_{n=1}^{\infty}\frac{q^{2n^2+2n}}{(q;q^2)_{n}(q;q^2)_{n+1}}.\label{gf:p1}
\end{align}
\end{corollary}

\begin{proof}
The length of the vertical (resp.~horizontal) strip that arises in the proof of Theorem~\ref{thm:bOE-srank-gf}
is precisely $\cE(\lambda)$ (resp.~$\cE(\lambda')$). Therefore, by setting $n=m$ and putting $z\to 1$,
$q\to q^{1/2}$ in \eqref{gf:boe}, we immediately arrive at \eqref{gf:p0}. Similarly, with $n=m+1$ we can
deduce \eqref{gf:p1}.
\end{proof}

\begin{remark}
Combining \eqref{def:psi} and \eqref{gf:p0}, we obtain $p_0(2n)\equiv p_\psi(n)\pmod{2}$ immediately.
Alternatively, this congruence could be combinatorially justified as we derive \eqref{cong-relate}.
\end{remark}

Let $p_\omega(n)$ denote the number of partitions of $n$ in which each odd part is less than
twice the smallest part. Just like the partition-theoretical interpretation for $p_{\nu}(n)$, it was first
noticed in the same paper by Andrews, Dixit, and Yee \cite[Theorem~3.1]{ADY} that
\begin{align*}
q\!\:\omega(q)=\sum_{n=1}^{\infty}p_{\omega}(n)q^n.
\end{align*}

\begin{corollary}\label{Cor:relat:mock-theta}
For any $n\geq1$,
\begin{align}
p_0(n)+p_1(n-1)+1=p_\omega(n).\label{relat-iden}
\end{align}
\end{corollary}

\begin{proof}
It follows form \eqref{gf:p0} and \eqref{gf:p1} that
\begin{align*}
 &\sum_{n=0}^\infty p_0(n)q^n+\sum_{n=0}^\infty p_1(n)q^{n+1}+\sum_{n=1}^\infty q^n\\
 &\quad=\sum_{n=0}^\infty\frac{q^{2n^2+4n+2}}{(q;q^2)_{n+1}^2}
+\sum_{n=1}^{\infty}\frac{q^{2n^2+2n+1}}{(q;q^2)_n(q;q^2)_{n+1}}+\dfrac{q}{1-q}\\
 &\quad=\sum_{n=0}^\infty\frac{q^{2n^2+4n+2}}{(q;q^2)_{n+1}^2}+
\sum_{n=1}^{\infty}\frac{q^{2n^2+2n+1}(1-q^{2n+1})}{(q;q^2)_{n+1}^2}+\dfrac{q}{1-q}\\
 &\quad=\sum_{n=1}^\infty\dfrac{q^{2n^2+2n+1}}{(q;q^2)_{n+1}^2}+\dfrac{q^2}{(1-q)^2}+\dfrac{q}{1-q}\\
 &\quad=\sum_{n=0}^\infty\dfrac{q^{2n^2+2n+1}}{(q;q^2)_{n+1}^2}
=q\!\:\omega(q)=\sum_{n=1}^\infty p_\omega(n)q^n,
\end{align*}
Then \eqref{relat-iden} follows by comparing the coefficients of $q^n$ on both sides of the identity above.
\end{proof}

\begin{?}
Prove \eqref{relat-iden} combinatorially via their partition-theoretical interpretations.
\end{?}


\section{Final remarks}

We conclude this paper with several conjectures to motivate further investigation.

\begin{conjecture}
Let $\ell>3$ be a prime number and $\ell\not\equiv23\pmod{24}$. Then for any $n\geq0$,
\begin{align*}
p_0\big(2\ell^2n+2\ell j+2\delta_\ell\big)\equiv0\pmod{4},
\end{align*}
where $0\leq j\leq\ell-2$ and $\delta_\ell$ is the least positive integer such that
$24\delta_\ell\equiv1\pmod{\ell}$.
\end{conjecture}

\begin{conjecture}
We have
\begin{align*}
\lim_{n\rightarrow\infty}\dfrac{\#\{m|\:p_0(m)\equiv0\pmod{2},\;\;0\leq m<n\}}{n}=\dfrac{1}{5}.
\end{align*}
\end{conjecture}

\begin{conjecture}
\leavevmode
\begin{enumerate}[{\rm (i)}]
\item For any $n\geq3$,
\begin{align*}
\overline{\mathcal{OE}}(2n+1)>\overline{\mathcal{EO}}(2n).
\end{align*}
\item For any $n\geq1$,
\begin{align*}
p_0(n)\ge p_1(n),
\end{align*}
where the strict inequality holds if $n\neq1$, $10$ or $13$.
\end{enumerate}
\end{conjecture}

\section*{Acknowledgements}

Shishuo Fu was partially supported by the National Natural Science Foundation of China (No. 12171059), the Natural
Science Foundation Project of Chongqing (No. cstc2021jcyj-msxmX0693), and the Mathematical Research Center of Chongqing 
University. Dazhao Tang was partially supported by the National Natural Science Foundation of China (No. 12201093), the 
Natural Science Foundation Project of Chongqing CSTB (No. CSTB2022NS\\CQ--MSX0387), the Science and Technology Research 
Program of Chongqing Municipal Education Commission (No. KJQN202200509), and the Doctoral start-up research grant 
(No. 21XLB038) of Chongqing Normal University.


\begin{thebibliography}{99}


\bibitem{and}
G.~E.~Andrews,
\textit{The Theory of Partitions}. Addison-Wesley Pub. Co., New York (1976).
Reissued, Cambridge University Press, New York, 1998.


\bibitem{and1}
G.~E.~Andrews,
On a partition function of Richard Stanley,
\textit{Electron. J. Combin.} \textbf{11} (2004), no. 2, R1.

\bibitem{and5}
G.~E.~Andrews,
Partitions, Durfee symbols, and the Atkin-Garvan moments of ranks,
\textit{Invent. Math.} \textbf{169} (2007), no. 1, 37--73.


\bibitem{and2}
G.~E.~Andrews,
Integer partitions with even parts below odd parts and the mock theta functions,
\textit{Ann. Comb.} \textbf{22} (2018), no. 3, 433--445.


\bibitem{and3}
G.~E.~Andrews,
Partitions with parts separated by parity,
\textit{Ann. Comb.} \textbf{23} (2019), no. 2, 241--248.


\bibitem{and4}
G. E. Andrews,
The Ramanujan--Dyson identities and George Beck's congruence conjectures,
\textit{Int. J. Number Theory} \textbf{17} (2021), no. 2, 239--249.


\bibitem{ADY}
G.~E.~Andrews, A.~Dixit, A.~J.~Yee,
Partitions associated with the Ramanujan/Watson mock theta functions
$\omega(q)$, $\nu(q)$ and $\phi(q)$,
\textit{Res. Number Theory} \textbf{1}, Paper No.~19, (2015), 25 pp.


\bibitem{APSY}
G. E. Andrews, D. Passary, J. A. Sellers, A. J. Yee,
Congruences related to the Ramanujan/Watson mock theta functions
$\omega(q)$ and $\nu(q)$,
\textit{Ramanujan J.} \textbf{43} (2017), no. 2, 347--357.


\bibitem{BG1}
A.~Berkovich, F.~G.~Garvan,
Dissecting the Stanley partition function,
\textit{J. Combin. Theory Ser. A.} {\bf 112} (2005), 277--291.


\bibitem{BG2}
A.~Berkovich, F.~G.~Garvan,
On the Andrews--Stanley refinement of Ramanujan's partition congruence
modulo $5$ and generalizations,
\textit{Trans. Amer. Math. Soc.} \textbf{358} (2005), no. 2, 703--726.


\bibitem{BO}
J. Bruinier, K. Ono,
Identities and congruences for the coefficients of Ramanujan's $\omega(q)$,
\textit{Ramanujan J.} \textbf{23} (2010), no. 1-3, 151--157.


\bibitem{CMO}
S. H. Chan, R. Mao, R. Osburn,
Variations of Andrews--Beck type congruences,
\textit{J. Math. Anal. Appl.} \textbf{495} (2021), no. 2, Paper No. 124771, 14 pp.


\bibitem{CG22}
R. Chen, F. G. Garvan,
A proof of the mod 4 unimodal sequence conjectures and related mock theta functions,
\textit{Adv. Math.} \textbf{398} (2022), Paper No. 108235, 50pp.


\bibitem{Che1}
S. Chern, Weighted partition rank and crank moments. I.
Andrews--Beck type congruences, to appear in
\textit{Proceedings of the Conference in Honor of Bruce Berndt}.


\bibitem{Che22a}
S. Chern,
Weighted partition rank and crank moments II. Odd-order moments,
\textit{Ramanujan J.} \textbf{57} (2022), no. 2, 471--485.


\bibitem{Che22b}
S. Chern,
Weighted partition rank and crank moments. III. A list of Andrews--Beck type
congruences modulo $5$, $7$, $11$ and $13$,
\textit{Int. J. Number Theory} \textbf{18} (2022), no. 1, 141--163.


\bibitem{CL}
S.~Corteel, J.~Lovejoy,
Overpartitions,
\textit{Trans. Amer. Math. Soc.} \textbf{356} (2004), no. 4, 1623--1635.


\bibitem{DT22a}
J. Q. D. Du, D. Tang,
Proofs of two conjectural Andrews--Beck type congruences due to Lin, Peng and Toh,
\textit{Int. J. Number Theory}, in press, https://doi.org/10.1142/S1793042123500677.


\bibitem{DT22b}
J. Q. D. Du, D. Tang,
Andrews--Beck type congruences for $k$-colored partitions, submitted.


\bibitem{Dys}
F. J. Dyson,
Some guesses in the theory of partitions,
\textit{Eureka} \textbf{8} (1944), 10--15.


\bibitem{Fu22}
S.~Fu, Combinatorial proofs and refinements of three partition theorems of Andrews,
\textit{Ramanujan J.} \textbf{60} (2023), 847--860.


\bibitem{FT}
S.~Fu, D.~Tang,
Partitions with fixed largest hook length,
\textit{Ramanujan J.} \textbf{45} (2018), no. 2, 375--390.


\bibitem{GKS}
F.~Garvan, D.~Kim, D.~Stanton,
Cranks and $t$-cores,
\textit{Invent. Math.} \textbf{101} (1990), no. 1, 1--17.


\bibitem{GP}
S. Garthwaite, D. Penniston,
$p$-adic properties of Maass forms arising from theta series,
\textit{Math. Res. Lett.} \textbf{15} (2008), no. 3, 459--470.


\bibitem{hir}
M.~D.~Hirschhorn,
Ramanujan's partition congruences,
\textit{Discrete Math.} \textbf{131} (1994), no. 1-3, 351--355.


\bibitem{hirb}
M.~D.~Hirschhorn,
\textit{The Power of $q$.} A personal journey.
Developments in Mahtematics Vol. 49, Springer, Cham (2017).


\bibitem{JLX22}
L. Jin, E. H. Liu, E. X. W. Xia,
Proofs of some conjectures of Chan--Mao--Osburn on Beck's partition statistics,
\textit{Rev. R. Acad. Cienc. Exactas F\'{\i}s. Nat. Ser. A Mat. RACSAM}
\textbf{116} (2022), no. 3, Paper No. 135, 13 pp.


\bibitem{KN}
W.~J.~Keith, R.~Nath,
Partitions with prescribed hooksets,
\textit{J. Comb. Number Theory} \textbf{3} (2011), no. 1, 39--50.


\bibitem{kim}
B.~Kim, Group actions on partitions,
\textit{Electron. J. Combin.} \textbf{24} (2017), no. 3, Paper No. 3.58, 11 pp.


\bibitem{Kim22}
E. Kim,
Andrews--Beck type congruences for overpartitions,
\textit{Electron. J. Combin.} \textbf{29} (2022), no. 1, Paper No. 1.37, 15 pp.


\bibitem{LPT21}
B. L. S. Lin, L. Peng, P. Toh,
Weighted generalized crank moments for $k$-colored partitions and Andrews--Beck type congruences,
\textit{Discrete Math.} \textbf{344} (2021), no. 8, Paper No. 112450, 13 pp.


\bibitem{LXY22}
Z.~Lin, H.~Xiong, S.~H.~F. Yan, Combinatorics of integer partitions with prescribed perimeter,
\textit{J. Combin. Theory Ser. A} \textbf{197} (2023), Paper No.~105747, 19 pp.


\bibitem{Mao22a}
R. Mao,
On total number parts functions associated to ranks of overpartitions,
\textit{J. Math. Anal. Appl.} \textbf{506} (2022), no. 2, Paper No. 125715, 16 pp.


\bibitem{Mao22b}
R. Mao,
On the total number of parts functions associated with ranks of partitions modulo 5 and 7,
\textit{Ramanujan J.} \textbf{58} (2022), no. 4, 1201--1243.


\bibitem{Mao23}
R. Mao,
Congruences for Andrews--Beck partition statistics modulo powers of primes,
\textit{Adv. in Appl. Math.} \textbf{146} (2023), Paper No. 102488, 19 pp.


\bibitem{Pas}
D. Passary, Studies of partition functions with conditions on parts and parity,
Thesis (Ph.D.) The Pennsylvania State University. 2019. 84 pp.


\bibitem{RB}
C.~Ray, R.~Barman,
On Andrews' integer partitions with even parts below odd parts,
\textit{J. Number Theory} \textbf{215} (2020), 321--338.


\bibitem{sag}
B.~E.~Sagan,
\textit{The Symmetric Group: Representations, Combinatorial Algorithms, and Symmetric Functions},
second edition, Grad. Texts in Math., vol. 203, Springer-Verlag, New York, 2001.


\bibitem{sta}
R.~P.~Stanley,
Some remarks on sign-balanced and maj-balanced posets,
\textit{Adv. Appl. Math.} \textbf{34} (2005), no. 4, 880--902.


\bibitem{Wang}
L. Wang,
Parity of coefficients of mock theta functions,
\textit{J. Number Theory} \textbf{229} (2021), 53--99.


\bibitem{Wat}
G. N. Watson,
The final problem: an account of the mock theta functions,
\textit{J. London Math. Soc.} \textbf{11} (1936), no. 1, 55--80.


\bibitem{Yao22}
O. X. M. Yao,
Proof of a Lin--Peng--Toh's conjecture on an Andrews--Beck type congruence,
\textit{Discrete Math.} \textbf{345} (2022), no. 1, Paper No. 112672, 7 pp.


\bibitem{Yeep}
A. J. Yee, Private communication.


\end{thebibliography}
\end{document}